\documentclass[12pt]{article}
\usepackage{amssymb}
\usepackage{amsthm}
\usepackage{amsmath}
\usepackage{graphicx}
\usepackage{enumerate}

\newtheorem{theorem}{Theorem}[section]
\newtheorem{definition}[theorem]{Definition}
\newtheorem{lemma}[theorem]{Lemma}

\newtheorem{example}[theorem]{Example}
\newtheorem{corollary}[theorem]{Corollary}
\title{Semimodular $\lambda$-lattices}
\author{Ivan~Chajda and Helmut~L\"anger}
\date{}
\begin{document}
\footnotetext[1]{Support of the research by \"OAD, project CZ~02/2019, and support of the research of the first author by IGA, project P\v rF~2019~015, is gratefully acknowledged.}
\maketitle
\begin{abstract}
The concept of a $\lambda$-lattice was introduced by V.~Sn\'a\v sel (\cite{Sn}) in order to generalize some lattice concepts for directed posets whose elements need not have suprema or infima. We extend the concept of semimodularity from lattices to $\lambda$-lattices and show connections to the lower covering condition and its generalizations. We further show that, contrary to the case of lattices, for $\lambda$-lattices semimodularity and the (weak) lower covering condition are independent properties. However, under some additional conditions semimodularity implies the (weak) lower covering condition. Examples of corresponding $\lambda$-lattices are presented.
\end{abstract}
 


\section{Introduction}

Posets are among the most frequently used relational structures in mathematics. In particular, lattices play an important role, i.e.\ posets where every two elements have a supremum and an infimum. Unfortunately, not every poset can be converted into a lattice. In Fig.~1 below there is depicted such a poset which is, moreover, bounded. Here all elements of $\mathbb R^+$ are under both elements $a$ and $b$, and these elements are under all elements of $\mathbb R^-$. (Here $\mathbb R^+$ and $\mathbb R^-$ denote the chains isomorphic to the poset of positive and negative reals, respectively.)

\vspace*{-8mm}

\[
\setlength{\unitlength}{7mm}
\begin{picture}(6,11)
\put(3,2){\circle*{.3}}
\put(3,3){\circle*{.3}}
\put(1,6){\circle*{.3}}
\put(5,6){\circle*{.3}}
\put(3,9){\circle*{.3}}
\put(3,10){\circle*{.3}}
\put(3,2){\line(0,1)2}
\put(3,8){\line(0,1)2}
\put(2.875,1.25){$0$}
\put(3.4,2.85){$c$}
\put(3,5){$\vdots$}
\put(.35,5.85){$a$}
\put(5.4,5.85){$b$}
\put(3,7){$\vdots$}
\put(3.4,8.85){$d$}
\put(2.85,10.4){$1$}
\put(2.2,.3){{\rm Fig.~1}}
\end{picture}
\]

\vspace*{-3mm}

It is evident that if $x,y\in\mathbb R^+\cup\mathbb R^-$ then $\sup(x,y)=\max(x,y)$ and $\inf(x,y)=\min(x,y)$ in this ordering. On the other hand, there do not exist $\sup(a,b)$ and $\inf(a,b)$. Thus this poset is not a lattice. However, if we use a choice function $\lambda$ which picks up some $c\in\mathbb R^+$ and $d\in\mathbb R^-$ arbitrarily (but fixed from now on) and we put
\begin{align*}
    a\vee b=b\vee a & :=d, \\
a\wedge b=b\wedge a & :=c, \\
            x\vee y & :=\max(x,y)\text{ if }(x,y)\neq(a,b),(b,a), \\
          x\wedge y & :=\min(x,y)\text{ if }(x,y)\neq(a,b),(b,a)
\end{align*}
then we obtain an algebra similar to a lattice, a so-called $\lambda$-lattice. This notion was introduced in \cite{Sn}. For an overview see \cite{CL}.

In a similar way, every directed poset (in particular, every bounded poset) can be organized into a $\lambda$-lattice. Surprisingly, $\lambda$-lattices can be described by relatively simple identities. These axioms are similar to those of lattices, only associativity of the operations is substituted by so-called weak associativity. However, the operations of $\lambda$-lattices need not be monotone. In fact, they are monotone if and only if the corresponding $\lambda$-lattice is a lattice. Namely, as pointed out above, for every two elements $a,b$ of a $\lambda$-lattice $\mathbf L=(L,\vee,\wedge)$ we have $a,b\leq a\vee b$. If $\vee$ is monotone and $a,b\leq c$ then $a\vee b\leq c\vee c=c$. Thus $a\vee b=\sup(a,b)$. Analogously, if $\wedge$ is monotone then $a\wedge b=\inf(a,b)$ and hence $\mathbf L$ is a lattice if both operations are monotone.

It turns out that $\lambda$-lattices are useful in some investigations concerning bounded posets since $\lambda$-lattices form a variety and thus the whole machinery of Universal Algebra can be applied.

\section{Semimodularity and lower covering conditions}

As mentioned above, the concept of a $\lambda$-lattice was introduced by V.~Sn\'a\v sel in \cite{Sn} as a non-associative generalization of a lattice. For the reader's convenience, we repeat this definition.

\begin{definition}
A {\em $\lambda$-lattice} is an algebra $(L,\vee,\wedge)$ of type $(2,2)$ satisfying the following identities:
\begin{enumerate}[{\rm(i)}]
\item $x\vee y\approx y\vee x$, $x\wedge y\approx y\wedge x$ {\rm(}commutativity{\rm)},
\item $x\vee((x\vee y)\vee z)\approx(x\vee y)\vee z$, $x\wedge((x\wedge y)\wedge z)\approx(x\wedge y)\wedge z$ {\rm(}weak associativity{\rm)},
\item $x\vee(x\wedge y)\approx x$, $x\wedge(x\vee y)\approx x$ {\rm(}absorption{\rm)}
\end{enumerate}
\end{definition}

Idempotency of $\vee$ and $\wedge$ follow easily by (iii):
\begin{align*}
  x\vee x & \approx x\vee(x\wedge(x\vee x))\approx x, \\
x\wedge x & \approx x\wedge(x\vee(x\wedge x))\approx x.
\end{align*}
If we define $x\leq y$ if $x\vee y=y$ then $x\leq y$ if and only if $x\wedge y=x$, and $(L,\leq)$ is a poset where $x\vee y\in U(x,y)$ and $x\wedge y\in L(x,y)$ for all $x,y\in L$. Here
\begin{align*}
U(x,y) & :=\{z\in L\mid x,y\leq z\}, \\
L(x,y) & :=\{z\in L\mid z\leq x,y\}.
\end{align*}
Also conversely, if $(L,\leq)$ is a {\em directed poset}, i.e.\ a poset satisfying $U(x,y)\neq\emptyset\neq L(x,y)$ for all $x,y\in L$ and one puts
\begin{align*}
& x\vee y\left\{
\begin{array}{ll}
:=\max(x,y) & \text{if }x\not\parallel y, \\
 \in U(x,y) & \text{otherwise},
\end{array}
\right. \\
& x\wedge y\left\{
\begin{array}{ll}
:=\min(x,y) & \text{if }x\not\parallel y, \\
 \in L(x,y) & \text{otherwise }
\end{array}
\right.
\end{align*}
such that $\vee$ and $\wedge$ are commutative then the resulting algebra $(L,\vee,\wedge)$ is a $\lambda$-lattice, see \cite{CL} and \cite{Sn} for details.

If a $\lambda$-lattice satisfies one of the distributivity or modularity laws then it is a lattice, see \cite{Sn}. Hence we are interested in weaker conditions which may be satisfied in $\lambda$-lattices which are not lattices. For this, we adopt the following concepts from \cite{Sz} (cf.\ also \cite{St}).

\begin{definition}
Let $\mathbf L=(L,\vee,\wedge)$ be a $\lambda$-lattice. We call $\mathbf L$ {\em semimodular} if for all $x,y,z\in L$ with $x\parallel y$ and $x\wedge y<z<x$ there exists some $u\in L$ with $x\wedge y<u\leq y$ and $(z\vee u)\wedge x=z$. We say that $\mathbf L$ satisfies the {\em weak lower covering condition} and the {\em lower covering condition} if
\begin{eqnarray}
x\wedge y\prec x\prec x\vee y & \text{implies} & y\prec x\vee y,\label{equ1} \\
             x\wedge y\prec x & \text{implies} & y\prec x\vee y\label{equ2}
\end{eqnarray}
{\rm(}$x,y\in L${\rm)}, respectively.
\end{definition}

It is clear that (\ref{equ1}) holds whenever $x\not\parallel y$ and that $x\wedge y\prec x\prec x\vee y$ implies $x\parallel y$. Therefore (\ref{equ1}) and (\ref{equ2}) are equivalent to
\begin{eqnarray*}
x\parallel y\text{ and }x\wedge y\prec x\prec x\vee y & \text{imply} & y\prec x\vee y, \\
             x\parallel y\text{ and }x\wedge y\prec x & \text{imply} & y\prec x\vee y
\end{eqnarray*}
{\rm(}$x,y\in L${\rm)}, respectively.

It is well-known (see e.g.\ \cite{Sz}) that every semimodular lattice satisfies the lower covering condition. Moreover, if a lattice is finite then it is semimodular if and only if it satisfies the lower covering condition. We are going to show that these relations do not hold for $\lambda$-lattices.

Now let us demonstrate the mentioned concepts by the following examples.

\begin{example}\label{ex1}
The $\lambda$-lattice $(L,\vee,\wedge)$ with $L:=\{0,a,b,c,d,e,1\}$, with the Hasse diagram visualized in Fig.~2

\vspace*{-8mm}

\[
\setlength{\unitlength}{7mm}
\begin{picture}(6,9)
\put(3,2){\circle*{.3}}
\put(1,4){\circle*{.3}}
\put(3,4){\circle*{.3}}
\put(5,4){\circle*{.3}}
\put(1,6){\circle*{.3}}
\put(3,6){\circle*{.3}}
\put(3,8){\circle*{.3}}
\put(3,2){\line(-1,1)2}
\put(3,2){\line(0,1)6}
\put(3,2){\line(1,1)2}
\put(1,6){\line(0,-1)2}
\put(1,6){\line(1,-1)2}
\put(1,6){\line(1,1)2}
\put(3,6){\line(-1,-1)2}
\put(3,6){\line(1,-1)2}
\put(2.875,1.25){$0$}
\put(.35,3.85){$a$}
\put(3.4,3.85){$b$}
\put(5.4,3.85){$c$}
\put(.35,5.85){$d$}
\put(3.4,5.85){$e$}
\put(2.85,8.4){$1$}
\put(2.2,.3){{\rm Fig.~2}}
\end{picture}
\]

\vspace*{-3mm}

and with $a\vee b=d$, $a\vee c=e$, $b\vee c=e$ and $d\wedge e=b$ satisfies the lower covering condition, but it is not semimodular since $d\parallel c$, $d\wedge c=0<a<d$, $c$ is the only element $x$ of $L$ satisfying $d\wedge c<x\leq c$, but $(a\vee c)\wedge d=e\wedge d=b\neq a$.
\end{example}

\begin{lemma}\label{lem1}
Let $(L,\vee,\wedge)$ be a semimodular $\lambda$-lattice and $a,b,c,d\in L$ and assume $a\parallel b$, $a\wedge b<c<a$, $a\wedge b<d<a$ and $c\neq d$. Then there exist $e,f\in L$ with $a\wedge b<e\leq b$, $a\wedge b<f\leq b$ and $c\vee e\neq d\vee f$.
\end{lemma}

\begin{proof}
According to semimodularity there exist $e,f\in L$ with $a\wedge b<e\leq b$, $(c\vee e)\wedge a=c$, $a\wedge b<f\leq b$ and $(d\vee f)\wedge a=d$. Now $c\vee e=d\vee f$ would imply
\[
c=(c\vee e)\wedge a=(d\vee f)\wedge a=d
\]
contradicting $c\neq d$.
\end{proof}

\begin{example}\label{ex6}
Consider the $\lambda$-lattice $\mathbf L$ from Example~\ref{ex1}. Then $d\parallel c$, $d\wedge c<a<d$, $d\wedge c<b<d$ and $a\neq b$. Since $c$ is the only element $x$ of $L$ satisfying $d\wedge c<x\leq c$ and since $a\vee c=e=b\vee c$, $\mathbf L$ is not semimodular according to Lemma~\ref{lem1}.
\end{example}

\begin{example}\label{ex4}
The $\lambda$-lattice from Example~\ref{ex1} with $a\vee b=1$ and all other results of meets and joins as in Example~\ref{ex1} is not semimodular {\rm(}this follows as in Example~\ref{ex1}{\rm)} and satisfies the weak lower covering condition, but not the lower covering condition since
\[
b\wedge a=0\prec b,\text{ but }a\not\prec1=b\vee a.
\]
This shows that the weak lower covering condition is strictly weaker than the lower covering condition.
\end{example}

\begin{example}\label{ex3}
The $\lambda$-lattice $(L,\vee,\wedge)$ with $L:=\{0,a,b,c,d,1\}$, with the Hasse diagram visualized in Fig.~3

\vspace*{-10mm}

\[
\setlength{\unitlength}{7mm}
\begin{picture}(6,9)
\put(3,2){\circle*{.3}}
\put(1,4){\circle*{.3}}
\put(5,4){\circle*{.3}}
\put(1,6){\circle*{.3}}
\put(5,6){\circle*{.3}}
\put(3,8){\circle*{.3}}
\put(3,2){\line(-1,1)2}
\put(3,2){\line(1,1)2}
\put(1,4){\line(0,1)2}
\put(1,4){\line(2,1)4}
\put(5,4){\line(-2,1)4}
\put(5,4){\line(0,1)2}
\put(3,8){\line(-1,-1)2}
\put(3,8){\line(1,-1)2}
\put(2.875,1.25){$0$}
\put(.35,3.85){$a$}
\put(5.4,3.85){$b$}
\put(.35,5.85){$c$}
\put(5.4,5.85){$d$}
\put(2.85,8.4){$1$}
\put(2.2,.3){{\rm Fig.~3}}
\end{picture}
\]

\vspace*{-3mm}

and with $a\vee b=c$ and $c\wedge d=a$ is semimodular and satisfies the lower covering condition.
\end{example}

\begin{example}\label{ex2}
The $\lambda$-lattice $(L,\vee,\wedge)$ with $L:=\{0,a,b,c,d,e,f,g,h,1\}$, with the Hasse diagram depicted in Fig.~4

\vspace*{-8mm}

\[
\setlength{\unitlength}{7mm}
\begin{picture}(6,9)
\put(3,2){\circle*{.3}}
\put(2,3){\circle*{.3}}
\put(4,3){\circle*{.3}}
\put(1,4){\circle*{.3}}
\put(3,4){\circle*{.3}}
\put(5,4){\circle*{.3}}
\put(1,6){\circle*{.3}}
\put(3,6){\circle*{.3}}
\put(5,6){\circle*{.3}}
\put(3,8){\circle*{.3}}
\put(3,2){\line(-1,1)2}
\put(3,2){\line(1,1)2}
\put(2,3){\line(1,1)3}
\put(4,3){\line(-1,1)3}
\put(1,4){\line(0,1)2}
\put(1,4){\line(1,1)2}
\put(3,4){\line(0,1)4}
\put(5,4){\line(-1,1)2}
\put(5,4){\line(0,1)2}
\put(3,8){\line(-1,-1)2}
\put(3,8){\line(1,-1)2}
\put(2.875,1.25){$0$}
\put(1.35,2.85){$a$}
\put(4.4,2.85){$b$}
\put(.35,3.85){$c$}
\put(3.4,3.85){$d$}
\put(5.4,3.85){$e$}
\put(.35,5.85){$f$}
\put(3.4,5.85){$g$}
\put(5.4,5.85){$h$}
\put(2.85,8.4){$1$}
\put(2.2,.3){{\rm Fig.~4}}
\end{picture}
\]

\vspace*{-3mm}

and with
\begin{align*}
  x\vee y & =\sup(x,y)\text{ if }\sup(x,y)\text{ is defined}, \\
x\wedge y & =\inf(x,y)\text{ if }\inf(x,y)\text{ is defined}
\end{align*}
{\rm(}$x,y\in L${\rm)} and $a\vee e=h$, $b\vee c=f$, $c\vee d=f$, $d\vee e=h$, $f\wedge g=c$ and $g\wedge h=e$ is semimodular and satisfies the lower covering condition.
\end{example}

There exist semimodular $\lambda$-lattices not satisfying the weak lower covering condition as the following examples show:

\begin{example}\label{ex7}
The $\lambda$-lattice $(L,\vee,\wedge)$ with $L:=\{0,a,b,c,d,1\}$, with the Hasse diagram visualized in Fig.~5

\vspace*{-8mm}

\[
\setlength{\unitlength}{7mm}
\begin{picture}(6,11)
\put(3,2){\circle*{.3}}
\put(3,4){\circle*{.3}}
\put(1,6){\circle*{.3}}
\put(5,7){\circle*{.3}}
\put(1,8){\circle*{.3}}
\put(3,10){\circle*{.3}}
\put(3,2){\line(0,1)2}
\put(3,4){\line(-1,1)2}
\put(3,4){\line(2,3)2}
\put(1,6){\line(0,2)2}
\put(3,10){\line(-1,-1)2}
\put(3,10){\line(2,-3)2}
\put(2.875,1.25){$0$}
\put(3.4,3.85){$a$}
\put(.35,5.85){$b$}
\put(5.4,6.85){$c$}
\put(.35,7.85){$d$}
\put(2.85,10.4){$1$}
\put(2.2,.3){{\rm Fig.~5}}
\end{picture}
\]

\vspace*{-3mm}

and with $b\wedge c=a$ and $c\wedge d=0$ is semimodular, but does not satisfy the weak lower covering condition since
\[
b\wedge c=a\prec c\prec1=b\vee c,\text{ but }b\not\prec1=b\vee c.
\]
This example together with Example~\ref{ex1} shows that semimodularity and the weak lower covering condition as well as semimodularity and the lower covering condition are independent.
\end{example}

\begin{example}\label{ex9}
Any $\lambda$-lattice $(L,\vee,\wedge)$ with $L:=\{0,a,b,c,d,1\}$, with the Hasse diagram depicted in Fig.~6

\vspace*{-8mm}

\[
\setlength{\unitlength}{7mm}
\begin{picture}(6,9)
\put(3,2){\circle*{.3}}
\put(2,3){\circle*{.3}}
\put(1,4){\circle*{.3}}
\put(5,4){\circle*{.3}}
\put(1,6){\circle*{.3}}
\put(5,6){\circle*{.3}}
\put(3,8){\circle*{.3}}
\put(3,2){\line(-1,1)2}
\put(3,2){\line(1,1)2}
\put(1,4){\line(0,1)2}
\put(1,4){\line(2,1)4}
\put(5,4){\line(-2,1)4}
\put(5,4){\line(0,1)2}
\put(3,8){\line(-1,-1)2}
\put(3,8){\line(1,-1)2}
\put(2.875,1.25){$0$}
\put(1.35,2.85){$a$}
\put(.35,3.85){$b$}
\put(5.4,3.85){$c$}
\put(.35,5.85){$d$}
\put(5.4,5.85){$e$}
\put(2.85,8.4){$1$}
\put(2.2,.3){{\rm Fig.~6  }}
\end{picture}
\]

\vspace*{-3mm}

and $a\vee c=d$ is not semimodular since $b\wedge c=0<a<b$, $c$ is the only element $x$ of $L$ satisfying $b\wedge c=0<x\leq c$, but $(a\vee c)\wedge b=d\wedge b=b\neq a$, and does not satisfy the weak lower covering condition since
\[
c\wedge a=0\prec c\prec d=c\vee a,\text{ but }a\not\prec d=c\vee a.
\]
\end{example}

However, if we add an appropriate condition then semimodularity implies the weak lower covering condition.

\begin{theorem}\label{th1}
Let $\mathbf L=(L,\vee,\wedge)$ be a semimodular $\lambda$-lattice satisfying the following condition:
\begin{equation}
\text{If }x,y,z\in L,\;x\parallel y,\;x\parallel z\text{ and }y<z\text{ then }x\wedge y\leq x\wedge z.\label{equ3}
\end{equation}
Then $\mathbf L$ satisfies the weak lower covering condition.
\end{theorem}

\begin{proof}
Let $a,b\in L$ and assume $a\parallel b$, $a\wedge b\prec a\prec a\vee b$ and $b\not\prec a\vee b$. Then there exists some $c\in L$ with $b<c<a\vee b$. Now $c\leq a$ would imply $b<c\leq a$ contradicting $a\parallel b$. On the other hand, $a<c$ would imply $a<c<a\vee b$ contradicting $a\prec a\vee b$. Hence $a\parallel c$. Because of $a\parallel b$, $a\parallel c$ and $b<c$ we have $a\wedge b\leq a\wedge c<a$ according to (\ref{equ3}) which together with $a\wedge b\prec a$ implies $a\wedge b=a\wedge c$. Now we have $c\parallel a$ and $c\wedge a=a\wedge b<b<c$. Because of semimodularity there exists some $d\in L$ with $c\wedge a<d\leq a$ and $(b\vee d)\wedge c=b$. From $c\wedge a=a\wedge b\prec a$ we conclude $d=a$. But then $c=(b\vee a)\wedge c=(b\vee d)\wedge c=b$ contradicting $b<c$.
\end{proof}

Let us note that condition (\ref{equ3}) does not imply monotonicity of $\wedge$. Namely the $\lambda$-lattice from Example~\ref{ex1} satisfies (\ref{equ3}), but $\wedge$ is not monotone since
\[
a\leq d,\text{ but }a\wedge e=a\not\leq b=b\wedge e.
\]

The situation described in Theorem~\ref{th1} was considered for so-called $\chi$-lattices in \cite R where instead of the weak lower covering condition the lower covering condition was considered. However, in contrast to $\lambda$-lattices in $\chi$-lattices joins and meets are minimal upper and maximal lower bounds, respectively.

It is elementary that a sublattice of a semimodular lattice need not be semimodular. This also holds for $\lambda$-lattices. However we can prove the following.

\begin{lemma}
Every convex sub-$\lambda$-lattice of a semimodular $\lambda$-lattice is semimodular.
\end{lemma}

\begin{proof}
If $\mathbf L=(L,\vee,\wedge)$ is a semimodular $\lambda$-lattice, $B$ a convex sub-$\lambda$-lattice of $\mathbf L$, $a,b,c\in B$, $a\parallel b$ and $a\wedge b<c<a$ then there exists some $d\in L$ with $a\wedge b<d\leq b$ and $(c\vee d)\wedge a=c$ and we have $d\in B$ because of the convexity of $B$.
\end{proof}

The question when a semimodular $\lambda$-lattice satisfies even the lower covering condition is answered in the following theorem.

\begin{theorem}
Let $\mathbf L=(L,\vee,\wedge)$ be a semimodular $\lambda$-lattice satisfying the following conditions:
\begin{eqnarray}
& & \text{If }x,y,z\in L,\;x\parallel y,\;x\parallel z\text{ and }y\prec z\text{ then }x\wedge y\leq x\wedge z,\label{equ4} \\
& & \text{if }x\parallel y,\;x<z\text{ and }y\prec z\text{ then }z\not<x\vee y,\label{equ5} \\
& & (L,\leq)\text{ satisfies the descending chain condition}.\label{equ6}
\end{eqnarray}
Then $\mathbf L$ satisfies the lower covering condition.
\end{theorem}

\begin{proof}
Let $a,b\in L$ and assume $a\parallel b$, $a\wedge b\prec a$ and $b\not\prec a\vee b$. Then, because of (\ref{equ6}), there exists some $c\in L$ with $b\prec c<a\vee b$. Now $c\leq a$ would imply $b<c\leq a$ contradicting $a\parallel b$. On the other hand, $a<c$ would imply $c\not<a\vee b$ according to (\ref{equ5}), a contradiction. This shows $a\parallel c$. Because of $a\parallel b$, $a\parallel c$ and $b\prec c$ we have $a\wedge b\leq a\wedge c<a$ according to (\ref{equ4}) which together with $a\wedge b\prec a$ implies $a\wedge b=a\wedge c$. Now we have $c\parallel a$ and $c\wedge a=a\wedge b<b<c$. Because of semimodularity there exists some $d\in L$ with $c\wedge a<d\leq a$ and $(b\vee d)\wedge c=b$. From $c\wedge a=a\wedge b\prec a$ we conclude $d=a$. But then $c=(b\vee a)\wedge c=(b\vee d)\wedge c=b$ contradicting $b<c$.
\end{proof}

\section{Heights of elements of $\lambda$-lattices}

For lattices (see \cite{St} and \cite{Sz}) as well as for $\chi$-lattices (\cite R) of finite length certain equalities and inequalities concerning the heights of elements of the form $a$, $b$, $a\vee b$ and $a\wedge b$ were derived. Analogous results are not possible for $\lambda$-lattices because the heights of $a\vee b$ and $a\wedge b$ do not depend on the heights of $a$ and $b$. However, we can prove the following result.

\begin{theorem}
Let $(L,\vee,\wedge,0,1)$ be a bounded $\lambda$-lattice of finite length satisfying the lower covering condition and $a,b\in L$ and assume
\[
a\not\parallel b\text{ or }a\wedge b\prec a\text{ or }a\wedge b\prec b.
\]
Then
\[
h(a\vee b)-h(a\wedge b)\leq|h(a)-h(b)|+2.
\]
\end{theorem}

\begin{proof}
If $a\leq b$ then
\[
h(a\vee b)-h(a\wedge b)=h(b)-h(a)\leq|h(a)-h(b)|+2.
\]
If $b\leq a$ then
\[
h(a\vee b)-h(a\wedge b)=h(a)-h(b)\leq|h(a)-h(b)|+2.
\]
If $a\wedge b\prec a$ then $b\prec a\vee b$ and hence
\begin{align*}
      h(a) & \leq h(a\wedge b)+1, \\
h(a\vee b) & \leq h(b)+1
\end{align*}
which implies
\[
h(a\vee b)-h(a\wedge b)\leq h(b)+1-h(a)+1\leq|h(a)-h(b)|+2.
\]
If, finally, $a\wedge b\prec b$ then $a\prec a\vee b$ and hence
\begin{align*}
      h(b) & \leq h(a\wedge b)+1, \\
h(a\vee b) & \leq h(a)+1
\end{align*}
which implies
\[
h(a\vee b)-h(a\wedge b)\leq h(a)+1-h(b)+1\leq|h(a)-h(b)|+2.
\]
\end{proof}

\section{Maximal chains in $\lambda$-lattices of finite length}

\begin{definition}\label{def1}
A poset $(P,\leq)$ is said to satisfy the {\em LU-covering property} if for every $x,y,z\in P$ with $x\prec y$, $x\prec z$ and $y\parallel z$ there exists some $u\in P$ with $y\prec u$ and $z\prec u$.
\end{definition}

Under the condition mentioned in Definition~\ref{def1} we can prove a result analogous to that for lattices, see e.g.\ \cite{St} and \cite{Sz}.

\begin{theorem}
Let $\mathbf P=(P,\leq,1)$ be a poset with $1$ of finite length satisfying the {\rm LU}-covering property. Then all maximal chains from a fixed element of $P$ to $1$ are of the same length.
\end{theorem}

\begin{proof}
We prove the following statement by induction on $n$: \\
If $a\in P$ and there exists some maximal chain from $a$ to $1$ of length $n$ then any maximal chain from $a$ to $1$ has length $n$. \\
For $n\leq1$ the statement is clear. Now assume $n>1$ and the statement to hold for all maximal chains from some fixed element of $P$ to $1$ of length $<n$. Let $a\in P$, $m\geq0$ and
\begin{align*}
a=a_0\prec\cdots\prec a_n & =1, \\
a=b_0\prec\cdots\prec b_m & =1
\end{align*}
be maximal chains from $a$ to $1$ of length $n$ and $m$, respectively. Since $n>1$ we have $m>1$. First assume $a_1=b_1$. Then
\begin{align*}
    a_1\prec\cdots\prec a_n & =1, \\
a_1=b_1\prec\cdots\prec b_m & =1
\end{align*}
are maximal chains from $a_1$ to $1$ of length $n-1$ and $m-1$, respectively. Because of the induction hypothesis we conclude $n-1=m-1$ and hence $n=m$. Now assume $a_1\neq b_1$. Then $a_1\parallel b_1$. Because of the LU-covering property there exists some $c_0\in P$ with $a_1\prec c_0$ and $b_1\prec c_0$. Since $\mathbf P$ is of finite length, there exists some maximal chain
\[
c_0\prec\cdots\prec c_k=1
\]
from $c_0$ to $1$ of length $k\geq0$. We then conclude that
\begin{align*}
         a_1\prec\cdots\prec a_n & =1, \\
a_1\prec c_0\prec\cdots\prec c_k & =1
\end{align*}
are maximal chains from $a_1$ to $1$ of length $n-1$ and $k+1$, respectively. According to the induction hypothesis, $n-1=k+1$. Analogously, one can show $m-1=k+1$. Together we obtain $n=k+2=m$.
\end{proof}

The following example shows that LU-covering condition does not imply that the involved $\lambda$-lattice is a lattice.

\begin{example}\label{ex5}
The poset from Example~\ref{ex1} satisfies the {\rm LU}-covering condition and all maximal chains from a fixed element of $L$ to $1$ have the same length.
\end{example}

\section{Acute $\lambda$-lattices}

The previous examples and reasoning lead us to introduce the following concept of an acute $\lambda$-lattice.

For every bounded poset $\mathbf P=(P,\leq,0,1)$ let $\mathbb L(\mathbf P)$ denote the $\lambda$-lattice $(P,\vee,\wedge)$ with $x\vee y=1$ and $x\wedge y=0$ for all $x,y\in P$ with $x\parallel y$. The $\lambda$-lattice arising in this way will be called the {\em acute $\lambda$-lattice} corresponding to $\mathbf P$.

In the following for every cardinal $k>1$ let ${\rm{\bf M}}_k$ denote the bounded lattice of length $2$ consisting of $0$ and $1$ and an antichain of cardinality $k$.

\begin{theorem}
Let $\mathbf P=(P,\leq,0,1)$ be a bounded poset and $A$ and $C$ denote the set of all atoms and coatoms of $(P,\leq)$, respectively. Then the following are equivalent:
\begin{enumerate}[{\rm(i)}]
\item The acute $\lambda$-lattice $\mathbb L(\mathbf P)$ corresponding to $\mathbf P$ satisfies the lower covering condition,
\item if $x\in A$, $y\in P$ and $x\parallel y$ then $y\in C$.
\item One of the following is true:
\begin{enumerate}[{\rm(a)}]
\item $A=\emptyset$,
\item $|A|=1$, and $x\leq y$ for all $x\in A$ and $y\in P\setminus\{0\}$,
\item there exists some cardinal $k>1$ with $\mathbf P\cong{\rm{\bf M}}_k$.
\end{enumerate}
\end{enumerate}
\end{theorem}

\begin{proof}
$\text{}$ \\
(i) $\Leftrightarrow$ (ii) is clear. \\
(ii) $\Rightarrow$ (iii): \\
First assume $|A|=1$. Let $a$ denote the unique element of $A$. Now assume there exists some $b\in P\setminus\{0\}$ with $a\not\leq b$. Then $b\neq1$ and $b\not\leq a$. Hence $a\parallel b$ which implies $b\in C$. Since $b\notin A\cup\{0\}$ there exists some $c\in P$ with $0<c<b$. Now $a\leq c$ would imply $a<b$, a contradiction, and $c<a$ would imply $c=0$, again a contradiction. Hence $a\parallel c$ and therefore $c\in C$ contradicting $c<b<1$. This shows $a\leq x$ for all $x\in P\setminus\{0\}$. \\
Now assume $|A|>1$. If $a,b\in A$ and $a\neq b$ then $a\parallel b$ and hence $a\in C$. This shows $A\subseteq C$. Now assume $C\setminus A\neq\emptyset$. Then there exists some $a\in C\setminus A$. Of course, $a\neq0$. Hence there exists some $b\in P$ with $0<b<a$. If $b\in A$ then $b\in C$ contradicting $b<a<1$. Hence $b\notin A$. Let $c\in A$. Then $b\parallel c$ would imply $b\in C$ contradicting $b<a<1$. Moreover, $b\leq c$ would imply $b\in\{0,c\}$, a contradiction. Therefore $c<b$ which implies $c\in A\subseteq C$ and $c<b<a$, again a contradiction. This shows $A=C$. Now assume $d\in P\setminus(A\cup\{0\})$ and $e\in A$. If $d\parallel e$ then $d\in C=A$, a contradiction. If $d\leq e$ then $d\in\{0,e\}\subseteq A\cup\{0\}$, again a contradiction. Hence $e<d$ and $e\in C$ whence $d=1$. \\
(iii) $\Rightarrow$ (ii) is clear.
\end{proof}

\begin{corollary}\label{cor1}
If $P$ is finite then the acute $\lambda$-lattice $\mathbb L(\mathbf P)$ corresponding to $\mathbf P$ satisfies the lower covering condition if and only if $|P|=1$ or $|A|=1$ or $\mathbf P\cong{\rm{\bf M}}_n$ with some integer $n>1$.
\end{corollary}

\begin{example}\label{ex8}
The acute $\lambda$-lattice $\mathbb L(\mathbf L)$ corresponding to the bounded poset $\mathbf L=(L,\leq,0,1)$ from Example~\ref{ex3} is semimodular and satisfies the weak lower covering condition, but it does not satisfy the lower covering condition according to Corollary~\ref{cor1}.
\end{example}

We can summarize some of our examples as follows. We use the following abbreviations:

SM $:=$ semimodularity, \\
WLCC $:=$ weak lower covering condition, \\
LCC $:=$ lower covering condition.
\[
\begin{tabular}{c|c|c|r}
SM  & WLCC & LCC &             Examples \\
\hline
yes & yes  & yes & \ref{ex3}, \ref{ex2} \\
yes & yes  & no  &            \ref{ex8} \\
yes &  no  & no  &            \ref{ex7} \\
no  & yes  & yes &            \ref{ex1} \\
no  & yes  & no  &            \ref{ex4} \\
no  &  no  & no  &            \ref{ex9}
\end{tabular}
\]

Authors' addresses:

Ivan Chajda \\
Palack\'y University Olomouc \\
Faculty of Science \\
Department of Algebra and Geometry \\
17.\ listopadu 12 \\
771 46 Olomouc \\
Czech Republic \\
ivan.chajda@upol.cz

Helmut L\"anger \\
TU Wien \\
Faculty of Mathematics and Geoinformation \\
Institute of Discrete Mathematics and Geometry \\
Wiedner Hauptstra\ss e 8-10 \\
1040 Vienna \\
Austria, and \\
Palack\'y University Olomouc \\
Faculty of Science \\
Department of Algebra and Geometry \\
17.\ listopadu 12 \\
771 46 Olomouc \\
Czech Republic \\
helmut.laenger@tuwien.ac.at
\end{document}